\documentclass{amsart}

\usepackage{graphicx}
\usepackage{amssymb}
\usepackage{float}
\usepackage{qtree}

\newtheorem{theorem}{Theorem}[section]
\newtheorem{lemma}[theorem]{Lemma}

\theoremstyle{definition}

\theoremstyle{remark}
\newtheorem{remark}[theorem]{Remark}

\numberwithin{equation}{section}

\newcommand{\Q}{{\mathbb Q}}
\newcommand{\Z}{{\mathbb Z}}

\newcommand{\Li}{\textnormal{Li}}
\newcommand{\rd}{{\rm d}}
\newcommand{\li}{\textnormal{li}}

\begin{document}

\title[Improved error bounds for the Fermat primality test]{Improved error bounds for the Fermat primality test on random inputs}

\author{Jared D. Lichtman}
\address{Department of Mathematics, Dartmouth College, Hanover, NH 03755}

\email{lichtman.18@dartmouth.edu}

\author{Carl Pomerance}
\address{Department of Mathematics, Dartmouth College, Hanover, NH 03755}
\email{carl.pomerance@dartmouth.edu}





\begin{abstract}
We investigate the probability that a random odd composite number passes a random
Fermat primality test, improving on earlier estimates in moderate ranges.  For 
example, with random numbers to $2^{200}$, our results improve
on prior estimates by close to 3 orders of magnitude.  
\end{abstract}

\maketitle


\section{Introduction}
Part of the basic landscape in elementary number theory is the Fermat congruence:
If $n$ is a prime and $1\le b\le n-1$, then
\begin{equation}
\label{eq:fermat}
b^{n-1}\equiv1\pmod n.
\end{equation}
It is attractive in its simplicity and ease of verification: using fast arithmetic subroutines,
\eqref{eq:fermat} may be checked in $(\log n)^{2+o(1)}$ bit operations.  
Further,
its converse (apparently) seldom lies.  In practice, if one has a large random
number $n$ that satisfies \eqref{eq:fermat} for a random choice for $b$, then almost
certainly $n$ is prime.  To be sure, there are infinitely many composites (the Carmichael numbers)
that satisfy \eqref{eq:fermat}
for all $b$ coprime to $n$, see \cite{AGP1}.  And in \cite{AGP2} it is shown that there are infinitely
many Carmichael numbers $n$ such that \eqref{eq:fermat} holds for $(1-o(1))n$ choices
for $b$ in $[1,n-1]$. 
(Specifically, for each fixed $k$ there are infinitely many Carmichael numbers $n$ such that the probability a random $b$ in $[1,n-1]$ has $(b,n)>1$ is less than $1/\log^k n$.)
 However, Carmichael numbers are rare, and if a number $n$ is chosen
at random, it is unlikely to be one.


We say $n$ is a {\em probable prime to the base} $b$ if \eqref{eq:fermat} holds.  A probable
prime is either prime or composite, but the terminology certainly suggests that it is probably
prime!  Specifically, let $P(x)$ denote the probability that an integer $n$ is composite given that
\begin{enumerate}
\item[(i)] $n$ is chosen at random with $1<n\le x$, $n$ odd,
\item[(ii)] $b$ is chosen at random with $1<b<n-1$, and
\item[(iii)] $n$ is a probable prime to the base $b$.
\end{enumerate}


It is known that if $x$ is sufficiently large, then $P(x)$ is small. Indeed,
Erd\H{o}s and Pomerance \cite[Theorem 2.2]{EP} proved that
\begin{equation} \label{eq:EP}
    P(x) \le \exp(-(1+o(1))\log x\log\log\log x / \log\log x)
\end{equation}
as $x\to\infty$. In particular, $\lim P(x) = 0$. Kim and Pomerance \cite{KP} replaced the asymptotic inequality of (\ref{eq:EP}) with the weaker, but explicit, inequality
$$
    P(x) \le (\log x)^{-197} \quad\text{ for }x\ge 10^{10^5}
$$
and gave numerical bounds on $P(x)$ for $10^{60}\le x < 10^{10^5}$. In this paper we simplify the argument in \cite{KP} and obtain better upper bounds on $P(x)$ for $10^{60}\le x \le 10^{90}$, as seen in Figure \ref{fig:upp}. 
In particular, at the start of this range, our bound is over 700 times smaller.

\begin{figure}[H]
  \caption{New bounds on $P(x)$.} \label{fig:upp}
  \[\begin{array}{ccccc}
      & \text{Bound on }          &    \text{New bound} \\
    x &P(x) \text{ in \cite{KP}} & \text{on } P(x) \\
    \hline
    10^{60} & 7.16E{-}2 & 1.002E{-}4\\
    10^{70} & 2.87E{-}3 & 1.538E{-}5\\
    10^{80} & 8.46E{-}5 & 2.503E{-}6\\
    10^{90} & 1.70E{-}6 & 4.304E{-}7\\
    10^{100} & 2.77E{-}8 & 7.798E{-}8
\end{array}\]
\end{figure}

\noindent The notation $aEm$ means $a\times10^{m}$.

With these methods, we also obtain new nontrivial bounds for $2^{40}\le x < 10^{60}$, values of $x$ smaller than the methods in \cite{KP} could handle. These results are included in Figure \ref{fig:uppp}. 

\begin{figure}[H]
  \caption{Upper bound on $P(2^k)$.} \label{fig:uppp}
  \[\begin{array}{cc|cc|cc}
    k &  P(2^k)\le & k &  P(2^k)\le & k & P(2^k)\le \\
    \hline
    40 & 4.306E{-}1 &
    140 & 3.265E{-}3 
& 240 & 1.017E{-}5 \\ 
50 & 2.904E{-}1 &
    150 & 1.799E{-}3 
& 250 & 5.876E{-}6 \\
    60 & 1.848E{-}1 &
    160 & 9.933E{-}4 
& 260 & 3.412E{-}6 \\
    70 & 1.127E{-}1 &
    170 & 5.505E{-}4 
& 270 & 1.992E{-}6 \\
    80 & 6.728E{-}2 &
    180 & 3.064E{-}4 
& 280 & 1.169E{-}6 \\
    90 & 4.017E{-}2
   & 190 & 1.714E{-}4 
& 290 & 6.888E{-}7 \\
    100 & 2.388E{-}2 
    & 200 & 9.634E{-}5 
& 300 & 4.080E{-}7\\
    110 & 1.435E{-}2 
   &  210 & 5.447E{-}5 
& 310 & 2.428E{-}7 \\
    120 & 8.612E{-}3 
& 220 & 3.097E{-}5 
& 320 & 1.451E{-}7 \\ 
    130 & 5.229E{-}3 
& 230 & 1.770E{-}5 
& 330 & 8.713E{-}8 \\  
\end{array}\]
\end{figure}

We compute the exact values of $P(x)$ for $x = 2^k$ with $3\le k \le 36$. Additionally, we estimate $P(x)$ for $x = 2^k$ with $30\le k \le 50$, using random sampling. Calibrating these estimates against the true values for $30\le k \le 36$ suggest that the estimates are fairly close to the true values for $37\le k \le 50$, and almost certainly within an order of magnitude from the truth.

A number $n$ is called $L$-smooth if all of its prime factors are bounded above by $L$.
The method of \cite{KP} first computes the contribution to $P(x)$ from numbers
that are not $L$-smooth (for an appropriate choice for $L$), and then enters a complicated
argument based on the asymptotic method of \cite{EP} for the contribution of the
$L$-smooth numbers.  In addition to small improvements made in the non-$L$-smooth
case, our principal new idea is to use merely that there are few $L$-smooth numbers.
For this we use the upper bound method pioneered
by Rankin in \cite{Rank} for this problem, obtaining numerically explicit upper bounds on sums over $L$-smooth numbers, c.f.\ equation \eqref{eq:first} and Remark \ref{rmk:Rankin}. These upper bounds should prove useful in other contexts.

One possible way to gain an improvement is to replace the Fermat test with the strong probable prime test of Selfridge.  Also known as the Miller--Rabin test, it is just as simple to perform and it returns fewer false positives. To describe this test, let $n>1$ be an odd number. First one computes $s,t$ with $n-1=2^s t$ and $t$ odd. Next, one chooses a number $b$, $1\le b\le n-1$. The number $n$ passes the test (and is called a \textit{strong probable prime to the base b}) if either
\begin{equation}
\label{eq:spp}
    b^t\equiv1\pmod{n}\quad\text{or}\quad b^{2^it}\equiv-1\pmod{n}\quad\text{for some }i<s. 
\end{equation}
Every odd prime must pass this test. Moreover, Monier \cite{M} and Rabin \cite{R} have shown that if $n$ is an odd composite, then the probability that it is a strong probable prime to a random base $b$ in $[1,n-1]$ is less than $\frac{1}{4}$.

Let $P_1(x)$ denote the same probability as $P(x)$, except that (iii) is replaced by
\begin{enumerate}
    \item[(iii)$'$] $n$ is a strong probable prime to the base $b$.
\end{enumerate}
Based on the Monier-Rabin theorem, one might assume that $P_1(x)\le \frac{1}{4}$, but as noted in \cite{B}, this reasoning is flawed.  However, in \cite{Bu} and \cite{DLP}, something similar to $P_1(x)\le\frac14$ is shown.
Namely, if $P_1'(2^k)$ is the analogous probability for odd $k$-bit integers, 
it is shown in \cite{Bu}, \cite{DLP} that $P_1'(2^k)\le\frac14$
for all $k\ge3$.   We show below how our estimates can be used to numerically
bound $P_1(x)$.  In particular, the results here improve on the estimates of
\cite{DLP} up to $2^{300}$.

\section*{Notation}
We have $(a,b)$, $[a,b]$ as the greatest common divisor, least common
multiple of the positive integers $a,b$, respectively.
We use $p$ and $q$ to denote prime numbers, and $p_i$ to denote the $i$th prime. For $n>1$, we let $P^+(n)$ denote the largest prime factor of $n$. 
Let $\varphi$ denote Euler's function, $\lambda$ the Carmichael universal exponent function, 
$\zeta$ the Riemann zeta-function,
$\Li(x) = \int_2^x \frac{\rd t}{\log t}$, and $\vartheta(x) = \sum_{p\le x}\log p$. In many instances, we take a sum over certain subsets of odd composite integers, in which cases we use $\sum'_n$ to denote $\sum_{\substack{n\text{ odd,}\\\text{composite}}}$.

\section{Preliminary lemmas}
In this section, we prove some preliminary lemmas which are needed for the rest of the paper, and which may be of interest in their own right.

\begin{lemma} \label{lma:S2I}
Given real numbers $a,b$ and a nonnegative, decreasing function $f$ on the interval $[a,b]$, we have that
$$\int_{\lceil a\rceil}^{b} f(t)\;\rd t\le\sum_{a\le n\le b} f(n) \le f(a) + \int_a^b f(t)\;\rd t.$$
\end{lemma}
The proof is clear.  Note that since
$\sum_{a< n\le b} f(n)\le\sum_{a\le n\le b} f(n)$, we may apply the upper bound for the sum on the half open interval.

\begin{lemma} \label{lma:phi/d}
For $x\ge 2$, we have that
$$\frac{x}{\zeta(2)} - \log x \le \sum_{n\le x} \frac{\varphi(n)}{n} \le \frac{x}{\zeta(2)} + \log x.$$
\end{lemma}
\begin{proof}
The result holds for $2\le x<18$, so assume $x\ge18$.
We have that
\begin{equation} \label{eq:lma}
\begin{split}
\sum_{n\le x} \frac{\varphi(n)}{n} & = \sum_{n\le x}\sum_{d\mid n}\frac{\mu(d)}{d} = \sum_{d\le x}\frac{\mu(d)}{d}\sum_{n\le x/d}1 = \sum_{d\le x}\frac{\mu(d)}{d}\left\lfloor\frac{x}{d}\right\rfloor\\
& = x\sum_{d\le x}\frac{\mu(d)}{d^2} - \sum_{d\le x}\frac{\mu(d)}{d}\left\{\frac{x}{d}\right\}\\
& = \frac{x}{\zeta(2)}-x\sum_{d> x}\frac{\mu(d)}{d^2} - \sum_{d\le x}\frac{\mu(d)}{d}\left\{\frac{x}{d}\right\},
\end{split}
\end{equation}
where $\{\ \}$ denotes the frational part.  By Lemma \ref{lma:S2I}, 
\begin{equation} \label{eq:lm}
\begin{split}
\sum_{d>x}\frac{\mu(d)}{d^2} & \le \sum_{d>x}\frac{1}{d^2} 
\le \frac{1}{x^2} + \int_x^\infty\frac{\rd t}{t^2} = \frac{1}{x^2} + \frac{1}{x},\\
\sum_{d>x}\frac{\mu(d)}{d^2} & \ge -\sum_{d>x}\frac{1}{d^2} \ge -\frac{1}{x^2} - \frac{1}{x}.\\
\end{split}
\end{equation}
Since $\sum_{1<d\le18,\,\mu(d)\neq-1}\frac1d=\frac{367}{336}>1.09$, we have
$$
-\sum_{d\le x}\frac{\mu(d)}{d}\left\{\frac{x}{d}\right\}  \le \sum_{\substack{d\le x\\\mu(d)=-1}}\frac{1}{d} \le \sum_{1<d\le x}\frac{1}{d} - \sum_{\substack{1<d\le 18\\\mu(d)\neq-1}}\frac{1}{d}
< \log x - 1.09.
$$
Substituting this and \eqref{eq:lm} back into \eqref{eq:lma} gives
\begin{equation*}
\sum_{n\le x} \frac{\varphi(n)}{n} \le \frac{x}{\zeta(2)} + \frac{1}{x} + 1 + \log x - 1.09 < \frac{x}{\zeta(2)} + \log x.
\end{equation*}
Similarly, direct computation shows that
$$
\sum_{d\le x}\frac{\mu(d)}{d}\left\{\frac{x}{d}\right\} \le \sum_{\substack{d\le x\\\mu(d)=1}}\frac{1}{d} \le \sum_{1<d\le x}\frac{1}{d} - \sum_{\substack{1<d\le 4\\\mu(d)\neq1}}\frac{1}{d}
< \log x -\frac{13}{12}.
$$
and thus
\begin{equation*}
\sum_{n\le x} \frac{\varphi(n)}{n} \ge \frac{x}{\zeta(2)} - \frac{1}{x} - 1 - \log x + \frac{13}{12} > \frac{x}{\zeta(2)} - \log x.
\end{equation*}
\end{proof}

\begin{lemma} \label{lma:phi/d2}
For $x\ge 1$, we have that
$$ \frac{\log x}{\zeta(2)} +1-\frac{\log2}{\zeta(2)} < \sum_{n\le x} \frac{\varphi(n)}{n^2} \le \frac{\log x}{\zeta(2)} + 1.$$
\end{lemma}
\begin{proof}
The inequalities are easily verified for $x<40$, so assume $x\ge40$.
 Partial summation gives
$$
\sum_{n\le x}\frac{\varphi(n)}{n^2}=\sum_{n\le39}\frac{\varphi(n)}{n^2}+\frac1x\sum_{n\le x}\frac{\varphi(n)}{n}-\frac1{40}\sum_{n\le 39}\frac{\varphi(n)}{n}+\int_{40}^x\frac1{t^2}\sum_{n\le t}\frac{\varphi(n)}{n}\,\rd t.
$$
Evaluating the two sums to 39 and using the upper and lower bounds in Lemma \ref{lma:phi/d} 
for the sums to $x$ and $t$, we obtain the stronger result,
$$ \frac{\log x}{\zeta(2)} +0.58 < \sum_{n\le x} \frac{\varphi(n)}{n^2}< \frac{\log x}{\zeta(2)} + 0.82.$$
Note that the upper bound in the lemma is tight at
$x=1$ and the lower bound cannot be improved as
$x\to2^-$.
\end{proof}

\begin{lemma} \label{lm:pc}
If $2\le y<x$ and $0<c<1$, then 
$$\sum_{y<p\le x} p^{-c} < f(x,y),$$
where 
$$
f(x,y):=(1+2.3\cdot10^{-8})\bigg(\Li(x^{1-c}) -\Li(y^{1-c}) + \frac{y^{1-c}}{\log y}\bigg) - \vartheta(y)\frac{y^{-c}}{\log y}.
$$
\end{lemma}
\begin{proof} 
We use the inequalities
\begin{equation}\label{eq:epnt}
\vartheta(x)<x \quad(0< x\le 10^{19}),\qquad|x-\vartheta(x)|<\epsilon x\quad(x>10^{19}),
\end{equation}
where $\epsilon=2.3\times10^{-8}$, see \cite{Bue2}, \cite{Bue}, improving on recent work in \cite{PT} (also  see \cite[Proposition 2.1]{LP}).
Let $f(t)=1/(t^c\log t)$.  By partial summation,
$$
\sum_{y<p\le x}p^{-c}=\sum_{y<p\le x}f(p)\log p
=\vartheta(x)f(x)-\vartheta(y)f(y)-\int_y^x\vartheta(t)f'(t)\,\rd t.
$$
Note that \eqref{eq:epnt} implies that $\vartheta(t)<(1+\epsilon)t$ for all $t>0$.
Since $f'(t)<0$ for $t\ge2$, we have
\begin{align*}
\sum_{y<p\le x}p^{-c}&<(1+\epsilon)xf(x)-(1+\epsilon)\int_y^xtf'(t)\,\rd t-\vartheta(y)f(y)\\
&=
(1+\epsilon)\left(\Li(x^{1-c})-\Li(y^{1-c})+yf(y)\right)-\vartheta(y)f(y),
\end{align*}
where we have integrated by parts and used that $\int f(t)\,dt=\Li(t^{1-c})$.  This completes the proof.
\end{proof}

\begin{lemma} \label{lma:ntail}
We have
\begin{align*}
    (i)\;\sum_{n>y}\frac{1}{n^2} & < \frac{5}{3y}&\text{for }y>0,\\
    (ii)\;\sum_{n \ge y}\frac{1}{n^3} &\le \frac{4(\zeta(3)-1)}{y^2}&\text{for }y>1.
\end{align*}
\end{lemma}
\begin{proof}
The first claim is stated and proved in (4.7) in \cite{KP}. We proceed similarly for the second claim.
When $1< y\le 2$, we have
\begin{align*}
    \sum_{n \ge y}\frac{1}{n^3} = \sum_{n \ge 2}\frac{1}{n^3} = \zeta(3) - 1 = \frac{4(\zeta(3)-1)}{4} \le \frac{4(\zeta(3)-1)}{y^2}.
\end{align*}
When $2< y\le 3$, direct computation shows that
\begin{align*}
    \sum_{n \ge y}\frac{1}{n^3} = \sum_{n \ge 3}\frac{1}{n^3} = \zeta(3) - 1 - \frac{1}{8}< \frac{4(\zeta(3)-1)}{y^2}.
\end{align*}
When $3< y\le 4$, direct computation shows that
\begin{align*}
    \sum_{n \ge y}\frac{1}{n^3} = \sum_{n \ge 4}\frac{1}{n^3} = \zeta(3) - 1 - \frac{1}{8} - \frac{1}{27} <\frac{4(\zeta(3)-1)}{y^2}.
\end{align*}
When $y>4$, by Lemma \ref{lma:S2I}, direct computation shows that
\begin{align*}
    \sum_{n \ge y}\frac{1}{n^3} \le \frac{1}{y^3} + \int_y^\infty \frac{\rd t}{t^3} = \frac{1}{y^3} + \frac{1}{2y^2}<\frac{4(\zeta(3)-1)}{y^2}.
\end{align*}

\end{proof}

\section{The basic method}
Let
$$\textbf{F}(n) = \{b\in(\Z/n\Z)^\times : b^{n-1}=1\}$$
and let $F(n)=\#\textbf{F}(n)$. If $n>1$ is odd, then $\pm1\in \textbf{F}(n)$. Thus, for these $n$, $F(n)-2$ counts the number of integers $b$, $1<b<n-1$, with $b^{n-1}\equiv1\pmod{n}$. Also note that by Fermat's little theorem, $F(p)=p-1$ for primes $p$. We thus have for $x\ge 5$,
\begin{equation}
P(x) = \frac{\sum'_{n\le x} (F(n)-2)}{\sum_{1<n\le x,\, n \text{ odd}} (F(n)-2)} = \bigg(1+\frac{\sum_{2<p\le x} (p-3)}{\sum'_{n\le x} (F(n)-2)} \bigg)^{-1}.
\end{equation}
Hence to obtain an upper bound for $P(x)$, we shall be interested in obtaining a lower bound for $\sum_{2<p\le x} (p-3)$ and an upper bound for $\sum'_{n\le x} (F(n)-2)$. To this end, we shall prove two theorems.
\begin{theorem} \label{thm:num}
For $x\ge 2657$, we have
$$\sum_{2<p\le x} (p-3) > \frac{x^2}{2\log x-\frac{1}{2}}.$$
\end{theorem}

\begin{theorem} \label{thm:denom}
Suppose $c$, $L_1$, and $L$ are arbitrary real numbers with $0<c<1$, 
$1<L_1<L$.  Then for any $x>L^2$, we have
$$\sideset{}{'}\sum_{n\le x} (F(n)-2) < x^{c+1}\prod_{2<p\le L}\big(1-p^{-c}\big)^{-1}+x^2B,$$
where 
\begin{align*}
B  &= 
\frac{1}{4L_1} +
\frac{\log L_1}{\zeta(2)}\Big(\frac1{2(L-1)}+\frac1{x^{1/2}}\Big)+\frac{.5}{L-1}+\frac{.8}{x^{1/2}}\\
&\quad+ \frac{L_1}{(x^{1/2}-1)^2}
+ \frac{(1+\log L_1)}{2(x^{1/2}-1)} + \frac{1}{(L-1)^2}\Big(\frac{L_1}{\zeta(2)} + \log L_1\Big).
\end{align*}
\end{theorem}

Before proving Theorems \ref{thm:num} and \ref{thm:denom}, we state the main result of the section, which follows from these theorems.
\begin{theorem} \label{thm:prob}
Suppose $c$, $L_1$, and $L$ are arbitrary positive real numbers satisfying $0<c<1$ and $1<L_1<L$. Then for any $x>\max\{L^2, 2657\}$, we have $P(x)\le 1/(1+z^{-1})$ where
$$ z = \bigg(B+x^{c-1}\prod_{2<p\le L}\big(1-p^{-c}\big)^{-1}\bigg)\big(2\log x - \tfrac{1}{2}\big),$$
and $B$ is defined as in Theorem \ref{thm:denom}.
\end{theorem}

In principle the prime sum is much larger than the composite sum, so the probability $P(x)$ may be approximately viewed as their quotient. We remark that the prime sum in Theorem \ref{thm:num} is asymptotically equal to $x^2/(2\log x)$, so the result is close to  best possible. Additionally, in the application of Theorem \ref{thm:denom}, $L$ and $c$ are used as parameters for smoothness and Rankin's upper bound, respectively.

We now prove Theorem \ref{thm:num} using 
 \eqref{eq:epnt} and the additional inequalities  from \cite{Bue2}, \cite{Bue} that
\begin{equation}
\label{eq:epnt2}
\vartheta(x)>x-2\sqrt{x}\quad(1423\le x\le 10^{19}),\qquad \pi(x)<(1+\epsilon)\li(x)\quad(x\ge2),
\end{equation}
where $\li(x)=\int_0^x\rd t/\log t$ and $\epsilon=2.3\times10^{-8}$.
\begin{proof}[Proof of Theorem \ref{thm:num}]
Let $A=1500$.  By partial summation,
\begin{equation} \label{eq:main}
\begin{split}
\sum_{2<p\le x} (p - 3) &= 1-3\pi(x) + \sum_{p\le x} p\\
&= 1 - 3\pi(x) +\sum_{2<p\le A}(p-3)+ \frac{x\vartheta(x)}{\log x}-\frac{A\vartheta(A)}{\log A} - \int_A^x \vartheta(t) \frac{\log t - 1}{\log^2 t}\;\rd t.
\end{split}
\end{equation}
By \eqref{eq:epnt2}, we have $-3\pi(x)>-3(1+\epsilon)\li(x)$.  Suppose that $A\le x\le 10^{19}$.  By \eqref{eq:epnt},
\eqref{eq:epnt2}, we have
\begin{align*}
\frac{x\vartheta(x)}{\log x}-\int_A^x\vartheta(t)\frac{\log t-1}{\log^2t}\,\rd t
&>\frac{x^2-2x^{3/2}}{\log x}-\int_A^x\frac t{\log t}-\frac t{\log^2t}\,\rd t\\
&=\li(x^2)-\frac{2x^{3/2}}{\log x}-\li(A^2)+\frac{A^2}{\log A}.
\end{align*}
Using these estimates in \eqref{eq:main}, we have
$$
\sum_{2<p\le x}(p-3)>\li(x^2)-\frac{2x^{3/2}}{\log x}-3(1+\epsilon)\li(x)+5\,875.
$$
It is now routine to verify
 the theorem for $17\,000\le x\le 10^{19}$.
Similar calculations with \eqref{eq:epnt}, \eqref{eq:epnt2} establish the theorem for $x>10^{19}$.
A simple check then verifies the theorem in the stated range.
\end{proof}

\section*{Proof of Theorem \ref{thm:denom}}

The bulk of the work is devoted to the proof of Theorem \ref{thm:denom}. The basic method is to divide the eligible $n$ into five parts, depending on the largest prime factor $P^+(n)$ as well as the quotient $\varphi(n)/F(n)$, indicating how close $n$ is to being a Carmichael number. We summarize this in the diagram below, which may help guide the reader through the proof.

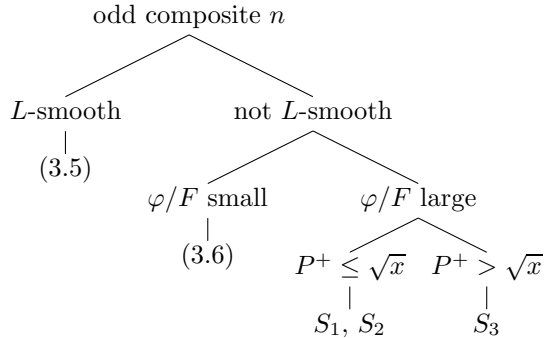
\Tree[.{odd composite $n$} [.$L$-smooth \eqref{eq:first} ]
                                              [.{not $L$-smooth} [.{$\varphi/F$ small} \eqref{eq:second} ]
                                                                              [.{$\varphi/F$ large} [.{$P^+\le\sqrt{x}$} $S_1,\,S_2$ ]      
                                                                                                               [.$P^+>\sqrt{x}$ $S_3$ ]]]]


For any $x>L^2$ with $L>L_1>1$, we have
\begin{equation} \label{eq:denom}
\begin{split}
\sideset{}{'}\sum_{n\le x} (F(n) - 2) & = \sideset{}{'}\sum_{\substack{n\le x\\P^+(n)\le L}} (F(n) - 2) + \sideset{}{'}\sum_{\substack{n\le x\\P^+(n)> L}} (F(n) - 2)\\
& \le \sum_{\substack{n\le x\\P^+(n)\le L\\n\textnormal{ odd}}} n + \sideset{}{'}\sum_{\substack{n\le x\\P^+(n)> L}} F(n).
\end{split}
\end{equation}
For the first term in (\ref{eq:denom}), we have for any $0<c<1$, 
\begin{equation} \label{eq:first}
\sum_{\substack{n\le x\\P^+(n)\le L\\2\nmid n}} n  \le x^{1+c} \sum_{\substack{P^+(n)\le L\\2\nmid n}}\frac{1}{n^c} = x^{1+c}\prod_{2<p\le L} \big( 1-p^{-c}\big)^{-1}.
\end{equation}

\begin{remark}\label{rmk:Rankin} By approximating the logarithm of the Euler product in 
\eqref{eq:first} (with 2 included) using Lemma \ref{lm:pc} and the
method of \cite{KP}, we can write a closed,
numerically explicit upper bound on the distribution
of $L$-smooth numbers:   If $\frac12<c<1$ and $37\le L< x$, then
$$
\sum_{\substack{n\le x\\P^+(n)\le L}}1\le x^cf_0\exp(A+f(L,36)),
$$
where the notation $f(a,b)$ is defined in Lemma \ref{lm:pc} and
$$
f_0:=\prod_{p<37}\big(1-p^{-c}\big)^{-1},
\quad A:=\frac{1}{2c-1}\Big(\frac{1}{2}+\frac{1}{3(37^c - 1)}\Big)\Big(36^{1-2c} - \frac{1}{2}\cdot37^{1-2c}\Big).
$$
There has been a very recent improvement of this Rankin-type upper bound due to
Granville and Soundarajan, see \cite[Theorem 5.1]{LP}, that is suitable for numerical
estimates.  It would be interesting to adapt that method to this paper.
\end{remark}

Now we bound the second term in (\ref{eq:denom}). Since $\textbf{F}(n)$ is a subgroup of $(\Z/n\Z)^\times$, by Lagrange's Theorem we have $F(n)\mid\varphi(n)$, where $\varphi$ is Euler's function. Then for each $k$, it makes sense to define $\textbf{C}_k(x)$ as the set of odd, composite $n\le x$ such that $F(n)=\varphi(n)/k$. Let $\textbf{C}'_k(x)$ be the set of $n\in \textbf{C}_k(x)$ for which $P^+(n)>L$, and let $C'_k(x)=\#\textbf{C}'_k(x)$. Thus, we have
\begin{equation} \label{eq:second}
\begin{split}
\sideset{}{'}\sum_{\substack{n\le x\\P^+(n)> L}} F(n) & = \sum_{k=1}^\infty \sum_{n\in \textbf{C}'_k(x)} F(n) = \sum_{k=1}^\infty \sum_{n\in \textbf{C}'_k(x)} \frac{\varphi(n)}{k}\\
& = \sum_{k\le L_1}\frac{1}{k} \sum_{n\in \textbf{C}'_k(x)} \varphi(n) + \sum_{k> L_1}\frac{1}{k} \sum_{n\in \textbf{C}'_k(x)} \varphi(n)\\
& \le x\sum_{k\le L_1}\frac{C'_k(x)}{k} + \frac{1}{L_1}\sum_{\substack{1<n\le x\\n\text{ odd}}} (n-2)
 \le x\sum_{k\le L_1}\frac{C'_k(x)}{k} + \frac{x^2}{4L_1}.
\end{split}
\end{equation}
It will thus be desirable to obtain an upper bound for $\sum_{k\le L_1}\frac{C'_k(x)}{k}$.
We remark that in the case $k>L_1$ we do not use $P^+(n)>L$; this observation will be
useful in the next section.

Given a prime $p>L$, $d\mid p-1$, let 
$$\mathbf S_{p,d}(x)=\{n:n\le x \text{ odd, composite},\;n\equiv p\kern-5pt\pmod{ p(p-1)/d}\}.$$
Let $S_{p,d}(x) = \#\mathbf S_{p,d}(x)$. Note that $S_{p,d}\le \frac{xd}{p(p-1)}$ by the Chinese Remainder Theorem.
We prove that
$$\bigcup_{k\le L_1} \mathbf C_k'(x) \subset \bigcup_{\substack{d\le L_1\\d\mid p-1\\L<p\le x}} \mathbf S_{p,d}(x).$$
Take $n$ in the left set. Then $p=P^+(n)>L$ and $k=\varphi(n)/F(n)\le L_1$. By Lemma 2.4 in \cite{KP}, we have 
$n\equiv1\pmod{\frac{p-1}{(k,p-1)}}$. Letting $d=(k,p-1)$, we have that $n\in \mathbf S_{p,d}$ 
(via the Chinese remainder theorem) and $d\le k\le L_1$, so $n$ is in the right set.

Additionally, for a given $p,d$ pair, $S_{p,d}$ counts integers $n = mp$ for which $m\equiv1\pmod{\frac{p-1}{d}}$. 
Write $m=1+u(\frac{p-1}{d})$ for some $u$. Letting $g=(u,d)$, we have that $m=1+(\frac{u}{g})(\frac{p-1}{d/g})$, 
so $n\in \mathbf S_{p,d/g}$, meaning that $n$ will be counted multiple times if $g>1$. 
Thus we require $(u,d)=1$. In particular, if $d$ is even, then $u$ is odd. 
Since $m=1+u(\frac{p-1}{d})$ is odd, we have $u(\frac{p-1}d)$ even. That is, if $d$ is even then $u$ is odd
and $\frac{p-1}d$ is even, so $2d\mid p-1$.  On the other hand, if $d$ is odd, we of course have $2d\mid p-1$.
Thus $2d\mid p-1$ always, and so
\begin{equation} \label{eq: tertiary}
\begin{split}
\sum_{k\le L_1}  \frac{C'_k(x)}{k} &\le \sum_{d\le L_1}\frac{1}{d}\sum_{\substack{L<p\le x\\2d\mid p-1}}\sum_{\substack{u\le \frac{xd}{p(p-1)}\\(u,d)=1}}1\\
& = \sum_{d\le L_1}\frac{1}{d}\sum_{\substack{L<p\le x^{1/2}\\2d\mid p-1}}\sum_{\substack{u\le \frac{xd}{p(p-1)}\\(u,d)=1}}1 + \sum_{d\le L_1}\frac{1}{d}\sum_{\substack{x^{1/2}<p\le x\\2d\mid p-1}}\sum_{\substack{u\le \frac{xd}{p(p-1)}\\(u,d)=1}}1\\
& < \sum_{d\le L_1}\frac{\varphi(d)}{d}\sum_{\substack{L<p\le x^{1/2}\\2d\mid p-1}}\Big(\frac{x}{p(p-1)}+1\Big) + \sum_{d\le L_1}\sum_{\substack{x^{1/2}<n\le x\\2d\mid n-1}}\frac{x}{n(n-1)}\\
&<S_1+S_2+S_3,
\end{split}
\end{equation}
where
\begin{equation}
\label{eq:defS}\begin{split}
S_1&=
 x\sum_{d\le L_1}\frac{\varphi(d)}{d}\sum_{\substack{L<n\le x^{1/2}\\2d\mid n-1}}\frac{1}{(n-1)^2},\quad
S_2=  \sum_{d\le L_1}\frac{\varphi(d)}{d}\sum_{\substack{1<n\le x^{1/2}\\2d\mid n-1}}1,\\
S_3&=  \sum_{d\le L_1}\sum_{\substack{x^{1/2}<n\le x\\2d\mid n-1}}\frac{x}{(n-1)^2}.
\end{split}\end{equation}
It is worth noting that in $S_1,S_2,S_3$, we have dropped the condition that $n$ be prime. An alternative bound using the condition of primality may be handled as an application of the Brun-Titchmarsh inequality. However, such a method is less effective for the small values of $x$ considered here.

Consider $S_1$ in \eqref{eq:defS}. For a given $d\le L_1$, by Lemma \ref{lma:S2I} we have that
\begin{equation}
\begin{split}
\label{eq:s1}
\sum_{\substack{n>L\\2d\mid n-1}}\frac{1}{(n-1)^2}  = \sum_{2du+1>L}\frac{1}{4d^2u^2} &\le \frac{1}{(L-1)^2} + \frac{1}{4d^2}\int_{(L-1)/2d}^\infty\frac{\rd t}{t^2}\\
& = \frac{1}{(L-1)^2} + \frac{1}{2d(L-1)}.
\end{split}
\end{equation}
Thus, by Lemma \ref{lma:phi/d} and Lemma \ref{lma:phi/d2}, 
\begin{equation} \label{eq:one}
\begin{split}
S_1 &< x\sum_{d\le L_1}\frac{\varphi(d)}{d} \Big( \frac{1}{(L-1)^2} + \frac{1}{2d(L-1)}\Big)\\
& \le \frac{x}{(L-1)^2}\big(\frac{L_1}{\zeta(2)} + \log L_1\big) + \frac{x}{2(L-1)}\Big(\frac{\log L_1}{\zeta(2)} + 1\Big).
\end{split}
\end{equation}

By Lemma \ref{lma:phi/d2}, $S_2$
in \eqref{eq:defS} is bounded by
\begin{equation} \label{eq:two}
S_2\le \sum_{d\le L_1}\frac{\varphi(d)}{d}\frac{x^{1/2}}{2d} \le x^{1/2}\Big(\frac{\log L_1}{2\zeta(2)} + .8\Big).
\end{equation}

We now consider $S_3$ in \eqref{eq:defS}. For a fixed $d\le L_1$, we have, as in \eqref{eq:s1},
\begin{equation*}
\sum_{\substack{x^{1/2}<n\le x\\2d\mid n-1}}\frac{1}{(n-1)^2} \le\frac1{(x^{1/2}-1)^2}+\frac1{2d(x^{1/2}-1)}.
\end{equation*}
So,
\begin{equation} \label{eq:three}
S_3\le
x\sum_{d\le L_1}\frac{1}{(x^{1/2}-1)^2} + \frac{1}{2d(x^{1/2}-1)}
 \le \frac{xL_1}{(x^{1/2}-1)^2} + \frac{x(1+\log L_1)}{2(x^{1/2}-1)}.
\end{equation}
By (\ref{eq:one}), (\ref{eq:two}), and (\ref{eq:three}), we obtain from \eqref{eq: tertiary} that
\begin{equation} \label{eq: tert}
\sum_{k\le L_1} \frac{C'_k(x)}{k} < x\Big(B-\frac{1}{4L_1}\Big)
\end{equation}
for $B$ as in Theorem \ref{thm:denom}. Thus, using (\ref{eq: tert}) in (\ref{eq:second}) gives the following result.
\begin{theorem} \label{thm:second}
Suppose $L_1$ and $L$ are arbitrary real numbers satisfying $1<L_1<L$. Then for any $x>L^2$, we have
$$\sideset{}{'}\sum_{\substack{n\le x\\P^+(n)> L}} (F(n) - 2) < x^2B.$$
where $B$ is as in Theorem \ref{thm:denom}.
\end{theorem}
Thus, (\ref{eq:denom}), \eqref{eq:first}, and Theorem \ref{thm:second} give us Theorem \ref{thm:denom}.

\section{A refinement of the basic method}
We refine the basic method as done analogously in \cite{KP}, by considering the \emph{two} largest prime factors of $n$. 
This refinement provides a modest improvement over Theorem \ref{thm:prob}
for $x$ starting around $2^{140}$.

\begin{theorem} \label{thm:Denom}
Suppose $c,L_1,L$, and $M$ are arbitrary real numbers satisfying $0<c<1$, $10<L_1<L$, $2L<M<L^2$.
 Then for any $x>L^2$, we have
\begin{align*}
\sideset{}{'}\sum_{n\le x} (F(n)-2) < x^{c+1}
\big(1+f(L,M^{1/2})\big)\prod_{2<p\le M^{1/2}}\big(1-p^{-c} \big)^{-1} + x^2(B+C),
\end{align*}
where $f$ is as in Lemma \ref{lm:pc}, $B$ is as in Theorem \ref{thm:denom}, and
\begin{align*}
C  = &\frac{L^2}{2x}(1+\log L_1) + \frac{2(1+\log L_1)^2}{M\log M}\\
& + \frac{1}{12(M-2L)}(1+\log L)(4+\log L_1)^4\Big(\frac{5}{12} + (\zeta(3)-1)(1+\log L) \Big).
\end{align*}
\end{theorem}
\begin{proof}
For each odd, composite $n\le x$, letting $P,Q$ be the two largest prime factors of $n$ (i.e. $P=P^+(n), Q=P^+(n/P)$), we have three possible cases,
\begin{enumerate}
    \item[(i)] $P>L$ or $F(n)<\varphi(n)/L_1$,
    \item[(ii)] $P\le L$ and $PQ\le M$,
    \item[(iii)] $P\le L$, $PQ>M$, and $F(n)\ge\varphi(n)/L_1$.
\end{enumerate}
It is worth noting that cases (i) and (ii) are not in general mutually exclusive. We retain Theorem \ref{thm:second} and the remark following \eqref{eq:second}
to handle case (i). For case (ii), let $0<c<1$. When $P\le M^{1/2}$, we have
\begin{align*}
\sum_{\substack{n\le x, 2\nmid n\\P\le M^{1/2}}} 1  
\le x^c \sum_{\substack{2\nmid n\\P^+(n)\le M^{1/2}}}n^{-c}.
\end{align*}
Similarly, when $P> M^{1/2}$ we have $Q \le \frac{M}{P} < M^{1/2}$, so
\begin{align*}
\sum_{\substack{n\le x, 2\nmid n\\M^{1/2}<P\le L\\Q\le M^{1/2}}} 1 & \le \sum_{M^{1/2}<p\le L}\sum_{\substack{n\le x/p\\P^+(n)\le M^{1/2}\\2\nmid n}} 1 \le \sum_{M^{1/2}<p\le L}\sum_{\substack{P^+(n)\le M^{1/2}\\2\nmid n}}\Big(\frac{x}{np}\Big)^c\\
& = x^c\sum_{M^{1/2}<p\le L}p^{-c} \sum_{\substack{2\nmid n\\P^+(n)\le M^{1/2}}}n^{-c}.
\end{align*}

Using Lemma \ref{lm:pc},
\begin{equation}
\label{eq:calc}
\begin{split}
\sum_{\substack{n\le x, 2\nmid n\\P\le L\\PQ \le M}} 1 & \le \sum_{\substack{n\le x, 2\nmid n\\P\le M^{1/2}}} 1 + \sum_{\substack{n\le x, 2\nmid n\\M^{1/2}<P\le L\\Q\le M^{1/2}}} 1\\
& \le x^c \sum_{\substack{2\nmid n\\P^+(n)\le M^{1/2}}}n^{-c} + x^c\sum_{M^{1/2}<p\le L}p^{-c} \sum_{\substack{2\nmid n\\P^+(n)\le M^{1/2}}}n^{-c}\\
& = x^c\bigg(1 + \sum_{M^{1/2}<p\le L}\kern-2pt p^{-c}\bigg)\kern-4pt\sum_{\substack{2\nmid n\\P^+(n)\le M^{1/2}}}\kern-5pt n^{-c}
\le x^c\big(1+f(L,M^{1/2})\big)\kern-8pt\sum_{\substack{2\nmid n\\P^+(n)\le M^{1/2}}}\kern-5pt n^{-c}\\
& = x^c\big(1+f(L,M^{1/2})\big)\prod_{2<p\le M^{1/2}}\big(1-p^{-c}\big)^{-1}.
\end{split}
\end{equation}
We now have the following result.
\begin{theorem} \label{thm:ii}
If $0<c<1$, $1<L<x$, and $L<M<L^2$, then
$$\sum_{\substack{n\le x, n\textnormal{ odd}\\P\le L\\PQ\le M}} n \le x^{c+1}\big(1+f(L,M^{1/2})\big)\prod_{2<p\le M^{1/2}}\big(1-p^{-c}\big)^{-1},$$
where $f$ is as in Lemma \ref{lm:pc}.
\end{theorem}

Consider $n$ belonging to case (iii). For each $k$, let $\mathbf B_k(x)$ denote the set of such $n$ with 
$\varphi(n)/F(n)=k$ and let $B_k(x) = \#\mathbf B_k(x)$. Thus, 
\begin{equation} \label{eq:iii}
\sideset{}{'}\sum_{n~{\rm in~case~(iii)}} F(n) \le x\sum_{k\le L_1}\frac{B_k(x)}{k}.
\end{equation}

By (2.11) in \cite{EP}, we have $\lambda(n)\mid k(n-1)$ for all $n\in \mathbf B_k(x)$. Since $PQ\mid n$, we have $\lambda(PQ)\mid \lambda(n)$, so $n$ satisfies the set of congruences
\begin{equation}\label{eq:cong}
n\equiv0\pmod{PQ},\quad k(n-1)\equiv0\pmod{\lambda(PQ)}.
\end{equation}
Suppose first that $P=Q$.  Then $\lambda(PQ)=P(P-1)$, so that \eqref{eq:cong} implies that $P\mid k$.
For such a prime $P$, the number of $n\le x$ with $P^2\mid n$ is at most $x/P^2<x/M$.  Thus, the contribution for
$n$
in this case is at most
\begin{equation}
\label{eq:q=p}
\frac xM\sum_{k\le L_1}\frac xk\sum_{\substack{P|k\\ P>M^{1/2}}}1<
\frac {x^2}M\Big(\sum_{k\le L_1}\frac1k\Big)\frac{\log L_1}{\log M^{1/2}} <\frac{2x^2}{M\log M}(1+\log L_1)^2.
\end{equation}
Now consider the case $P>Q$.
The latter congruence in \eqref{eq:cong} is equivalent to
\begin{equation*}
    n \equiv 1 \quad\bigg(\text{mod }\frac{\lambda(PQ)}{(k,\lambda(PQ))}\bigg),
\end{equation*}
from which we also note
\begin{equation*}
    \bigg(PQ\,,\,\frac{\lambda(PQ)}{(k,\lambda(PQ))}\bigg)=1.
\end{equation*}
Thus for arbitrary fixed primes $p>q$, the Chinese remainder theorem gives that the number of integers $n\le x$ satisfying the system $n\equiv0\pmod{pq}$, $k(n-1)\equiv0\pmod{\lambda(pq)}$ as in \eqref{eq:cong} is at most
\begin{equation*}
    1 + \frac{x(k,\lambda(pq))}{pq\lambda(pq)}.
\end{equation*}
Summing over choices for $p,q$, we have the number of $n$ in this case is at most
\begin{equation} \label{eq:q<p}
\sum_{\substack{q< p \le L\\pq>M}} \bigg(1 + \frac{x(k,\lambda(pq))}{pq\lambda(pq)}\bigg) \le \frac{1}{2}L^2 + \frac{1}{2}x\sum_{\substack{p,q \le L\\pq>M\\p\neq q}} \frac{(k,[p-1,q-1])}{pq[p-1,q-1]}.
\end{equation}
This is (4.4) in \cite{KP} where ``$L_2$" there is our ``$L$". 
Following the argument in \cite{KP} from there, and letting $M'=M-2L$ and with
$u_1,u_2,u_3,u_4,\mu,\nu,\delta$ positive integer variables, we have that
\begin{equation} \label{eq:dmove}
    \sum_{\substack{q, p \le L\\pq>M\\p\neq q}}  \frac{(k,[p-1,q-1])}{pq[p-1,q-1]} \le 
\sum_{\substack{u_1u_2u_3u_4=k\\(u_1,u_2)=1}}\kern2pt\sum_{\substack{\mu\le L/u_1\\\nu\le L/u_2}} \kern2pt\sum_{u_1u_2u_3^2\mu\nu\delta^2 > M'} \frac{1}{\mu^2\nu^2\delta^3u_1u_2u_3^2}.
\end{equation}
which is the initial inequality of (4.6) in \cite{KP} and with a typo corrected (the variable ``$\delta$" under
the second summation there should be ``$\mu$").

We now diverge from the argument in \cite{KP}, and split up the sum on the right side of \eqref{eq:dmove} into two cases, $\delta=1$ and $\delta\ge 2$. When $\delta=1$, by Lemma \ref{lma:ntail}(i) we have
\begin{equation} \label{eq:d=1}
\begin{split}
    \sum_{\substack{u_1u_2u_3u_4=k\\(u_1,u_2)=1}} \kern2pt \sum_{\substack{\mu\le L/u_1\\\nu\le L/u_2}}\kern2pt\sum_{\mu\nu u_1u_2u_3^2 > M'} \frac{1}{\mu^2\nu^2u_1u_2u_3^2} 
&< \frac{5}{3M'}\sum_{u_1u_2u_3u_4=k}\kern2pt\sum_{\nu\le L/u_2}\frac{1}{\nu}\\
    & \le \frac{5}{3M'}(1+\log L)\sum_{u_1u_2u_3u_4=k}1,
\end{split}
\end{equation}
When $\delta\ge 2$, let $D:= \sqrt{M'/u_1u_2u_3^2\mu\nu}$. By Lemma \ref{lma:ntail}(ii) we have
\begin{equation} \label{eq:dge2}
\begin{split}
    \sum_{\substack{u_1u_2u_3u_4=k\\(u_1,u_2)=1}}\sum_{\substack{\mu\le L/u_1\\\nu\le L/u_2}} \sum_{\delta \ge \max\{2,D\}} \frac{1}{\mu^2\nu^2\delta^3u_1u_2u_3^2}
     &\le \frac{4(\zeta(3)-1)}{M'}\sum_{u_1u_2u_3u_4=k}\sum_{\substack{\mu\le L/u_1\\\nu\le L/u_2}}\frac{1}{\mu\nu}\\
    & \le \frac{4(\zeta(3)-1)}{M'}(1+\log L)^2\sum_{u_1u_2u_3u_4=k}1.
\end{split}
\end{equation}

Substituting \eqref{eq:d=1} and \eqref{eq:dge2} back into \eqref{eq:dmove}
and then \eqref{eq:q<p}, we have
\begin{equation}
\label{eq:justabove}
\begin{split}
    \sum_{k\le L_1}\frac{1}{k} & \sum_{\substack{q< p \le L\\pq>M}} \bigg(1 + \frac{x(k,\lambda(pq))}{pq\lambda(pq)}\bigg) < \frac{1}{2}L^2(1+\log L_1)\\
    & + x(1+\log L)\Big(\frac{5}{6M'} + \frac{2(\zeta(3)-1)}{M'}(1+\log L) \Big)\sum_{k\le L_1}\frac{\tau_{(4)}(k)}{k},
\end{split}
\end{equation}
where $\tau_{(i)}(k)$ is the number of ordered factorizations of $k$ into $i$ positive factors. 
In \cite{KP} (see (4.9)), an easy induction argument shows that
\begin{equation*}
    \sum_{k\le y}\frac{\tau_{(i)}(k)}{k} \le \frac{1}{i!}(i+\log y)^i
\end{equation*}
for any natural number $i$ and any $y\ge 1$.  Using this in \eqref{eq:justabove}
and then combining with \eqref{eq:q=p} gives
\begin{equation*}
x\sum_{k\le L_1}\frac{B_k(x)}{k} \le x^2C,
\end{equation*}
where $C$ is
as in Theorem \ref{thm:Denom}.  
Thus, from (\ref{eq:iii}) we have the following result.
\begin{theorem} \label{thm:iii}
If $10<L_1<L<M/2$ and $x>L^2>M$, then
$$\sideset{}{'}\sum_{n~{\rm in~case~(iii)}}F(n) \le x^2C,$$
where $C$ is as in Theorem \ref{thm:Denom}.
\end{theorem}
Combining Theorems \ref{thm:second}, \ref{thm:ii} and \ref{thm:iii} yield Theorem \ref{thm:Denom}.
\end{proof}

Finally, Theorems \ref{thm:num} and \ref{thm:Denom} give the following result.
\begin{theorem} \label{thm:Prob}
If $0<c<1$, $10<L_1<L$, $2L<M<L^2<x$, and $x\ge2657$, then $P(x)\le 1/(1+z^{-1})$ where
$$ z =\bigg( x^{c-1}\big(1+f(L,M^{1/2})\big)\prod_{2<p\le M^{1/2}}\big(1-p^{-c}\big)^{-1}+B+C\bigg)
\big(2\log x - \tfrac{1}{2}\big),$$
$f$ is as in Lemma \ref{lm:pc}, $B$ is as in Theorem \ref{thm:denom}, and $C$ is as in Theorem \ref{thm:Denom}.
\end{theorem}

\section{The strong probable prime test}

The next theorem
extends the applicability of Theorems \ref{thm:prob} and \ref{thm:Prob} to the probability, $P_1(x)$, that an odd composite $n\le x$ passes the strong probable prime test to a random base. 
For an odd number $n$, let  $S(n)$ denote the number of integers
$1\le b\le n-1$ such that $n$ is a strong probable prime to the base
$b$, cf.\ \eqref{eq:spp}.  Thus,
$$
P_1(x)=\frac{\sideset{}{'}\sum_{\kern-3pt n\le x}\big(S(n)-2\big)}
{\sideset{}{'}\sum_{\kern-3pt n\le x}\big(S(n)-2\big)+\sideset{}{}\sum_{2<p\le x}(p-3)}.
$$
The following theorem together with Theorems \ref{thm:num}, \ref{thm:denom},
and \ref{thm:Denom} allows for a numerical estimation of $P_1(x)$ for various
values of $x$.

\begin{theorem}\label{thm:p01}
For $x\ge 1$, we have that 
$$
\sideset{}{'}\sum_{n\le x}\big(S(n)-2\big)\le
\frac12\sideset{}{'}\sum_{n\le x}\big(F(n)-2\big).
$$
\end{theorem}
\begin{proof}
By (2.1) in \cite{DLP}, we have that $S(n) \le 2^{1-\omega(n)}F(n)$, where $\omega(n)$ denotes the number of distinct prime factors of $n$. 
So, if $n$ is odd and divisible by at least 2 different primes, we have $S(n)\le\frac12F(n)$.  Further,
if $n=p^a$ is an odd prime power then $S(p^a) = F(p^a) = p-1$. 
Therefore we have
\begin{align*}
\sideset{}{'}\sum_{n\le x} \big(S(n)-2\big) & \le \sideset{}{'}\sum_{n\le x} \Big(\frac{1}{2}F(n)-2\Big) + \frac{1}{2}\sum_{\substack{2<p^a\le x\\a\ge2}}(p-1)\\
&=\frac12\sideset{}{'}\sum_{n\le x}\big(F(n)-2\big)-\sideset{}{'}\sum_{n\le x}1+\frac12\sum_{\substack{2<p\le x^{1/a}\\a\ge2}}(p-1),
\end{align*}
so to prove the theorem it is enough to show that
\begin{equation}
\label{eq:suff}
\sideset{}{'}\sum_{n\le x}1\ge\frac12\sum_{\substack{2<p\le x^{1/a}\\a\ge2}}(p-1).
\end{equation}
Since 3 times an odd integer $>1$ is an odd composite number, we have
$$
\sideset{}{'}\sum_{n\le x}1\ge\sum_{\substack{1<m\le x/3\\m\,{\rm odd}}}1=\left\lfloor\frac x6-\frac12\right\rfloor
>\frac16x-\frac32.
$$
Also, since the primes larger than 2 are odd, for a given value of $a$ we have
$$
\frac12\sum_{2<p\le x^{1/a}}(p-1)\le\sum_{j\le \frac12(x^{1/a}-1)}j\le\frac12\Big(\frac12x^{1/a}-1\Big)\Big(\frac12x^{1/a}+1\Big)<\frac18x^{2/a}.
$$
Adding these inequalities for $a=2,3,\dots,\lfloor\log x/\log 3\rfloor$, we see that \eqref{eq:suff}
will follow if we  show that
$$
\frac16x-\frac32>\frac18x+\frac18x^{2/3}+\frac18x^{1/2}(\log x/\log3-3).
$$
This inequality holds for $x\ge254$.
 For $9\le x<254$, \eqref{eq:suff} can be verified directly.
Indeed, the prime sum in \eqref{eq:suff} increases only at the 8 powers of
odd primes to 254 and it is enough to compute the two sums at those points.
For $x<9$,
$$\sideset{}{'}\sum_{n\le x} \big(F(n)-2\big)=\sideset{}{'}\sum_{n\le x} \big(S(n)-2\big)=0 ,$$
so the theorem holds here as well.
This completes the proof.
\end{proof}

We remark that the same result holds for the Euler probable prime test (also known as  the Solovay--Strassen
test).  This involves verifying that the odd number $n$ satisfies $a^{(n-1)/2}\equiv\left(\frac an\right)\pmod n$, where $\left(\frac an\right)$ is the Jacobi symbol.  Indeed, from Monier's formula, see \cite[(5.4)]{EP}, we
have that the number of bases $a$ (mod $n$) for which the Euler congruence holds is also $\le 2^{1-\omega(n)}F(n)$.
Like the strong test (as discussed in the introduction), an advantage with the Euler probable prime test is that more liars may be weeded out
by repeating the test.

\section{Numerical results}
We apply Theorems \ref{thm:prob} and \ref{thm:Prob} to obtain numerical bounds on $P(x)$ for various values of $x$. 
In Figure \ref{fig:up}, bounds on $P(2^k)$ are computed via Theorem \ref{thm:prob} for $40\le k \le 130$ and 
Theorem \ref{thm:Prob} for $140\le k \le 330$, at which point the methods of this paper lose their edge over those in \cite{KP}. To select values for parameters $L,L_1, M, c$, we started with an initial guess based on \cite{KP}, and then optimized each parameter in turn (holding the others fixed). The reported values were determined by repeated this process five times.

Note that the upper bounds in Theorems \ref{thm:prob}, \ref{thm:Prob} are
decreasing functions in $x$, so one can use the Figure \ref{fig:up} data
to compute upper bounds for values of $x$ between consecutive entries. 

We also compute the exact values of $P(x)$ for $x=2^k$ when $k\le 36$. By definition,
$$P(x) = \frac{S_c(x)}{S_c(x)+S_p(x)}$$
for
\begin{equation*}
    S_p(x) = \sum_{2<p\le x}(p-3),\quad S_c(x) = \sideset{}{'}\sum_{n\le x}\big(F(n)-2\big).
\end{equation*}
For ease, we have split up the computation into dyadic intervals $(2^{k-1},2^k)$. Letting 
\begin{equation*}
    S_p(x/2,x) = \sum_{x/2<p\le x}(p-3),\quad S_c(x/2,x) = \sideset{}{'}\sum_{x/2<n\le x}\big(F(n)-2\big),
\end{equation*}
we have that
\begin{equation}\label{eq:S2P}
P(2^k) = \frac{\sum_{j=3}^k S_c(2^{j-1},2^j)}{\sum_{j=3}^k S_p(2^{j-1},2^j)+ S_c(2^{j-1},2^j)}.
\end{equation}

\begin{figure}[H]
  \caption{Upper bound on $P(2^k)$.} \label{fig:up}
  \[\begin{array}{rccccc}
    k & L & L_1 & M^{1/2} & c &  P(2^k)\le \\
    \hline
    40 & 307^{-}      & 135        & & 0.5440 & 4.306E{-}1 \\
    50 & 727^{-} & 318 & & 0.5831 & 2.904E{-}1 \\
    60 & 1.860E{+}3 & 831 & & 0.6235 & 1.848E{-}1 \\
    70 & 4.000E{+}3 & 1.75E{+}3 & & 0.6491 & 1.127E{-}1 \\
    80 & 8.500E{+}3 & 3.72E{+3} & & 0.6704 & 6.728E{-}2 \\
    90 & 1.804E{+}4 & 7.55E{+}3 & & 0.6906 & 4.017E{-}2 \\
    100 & 3.505E{+}4 & 1.54E{+}4 & & 0.7052 & 2.388E{-}2 \\
    110 & 7.351E{+}4 & 3.27E{+}4 & & 0.7217 & 1.435E{-}2 \\
    120 & 1.354E{+}5 & 5.95E{+}4 & & 0.7321 & 8.612E{-}3 \\
    130 & 2.507E{+}5 & 1.10E{+}5 & & 0.7423 & 5.229E{-}3 \\
    140 & 9.90E{+}5 & 1.57E{+}5 & 2.379E{+}5 & 0.7444 & 3.265E{-}3\\ 
    150 & 2.20E{+}6 & 3.19E{+}5 & 3.739E{+}5 & 0.7504 & 1.799E{-}3\\ 
    160 & 4.88E{+}6 & 6.21E{+}5 & 5.689E{+}5 & 0.7554 & 9.932E{-}4\\ 
    170 & 1.05E{+}7 & 1.21E{+}6 & 8.669E{+}5 & 0.7602 & 5.505E{-}4\\ 
    180 & 2.21E{+}7 & 2.30E{+}6 & 1.315E{+}6 & 0.7648 & 3.064E{-}4\\ 
    190 & 4.55E{+}7 & 4.55E{+}6 & 1.990E{+}6 & 0.7692 & 1.714E{-}4\\ 
    200 & 9.23E{+}7 & 8.69E{+}6 & 2.990E{+}6 & 0.7734 & 9.634E{-}5\\ 
    210 & 1.84E{+}8 & 1.66E{+}7 & 4.455E{+}6 & 0.7773 & 5.447E{-}5\\ 
    220 & 3.62E{+}8 & 3.16E{+}7 & 6.627E{+}6 & 0.7811 & 3.097E{-}5\\ 
    230 & 7.19E{+}8 & 5.74E{+}7 & 9.644E{+}6 & 0.7845 & 1.770E{-}5\\ 
    240 & 1.38E{+}9 & 1.09E{+}8 & 1.410E{+}7 & 0.7878 & 1.017E{-}5\\ 
    250 & 2.62E{+}9 & 2.01E{+}8 & 2.049E{+}7 & 0.7911 & 5.876E{-}6\\ 
    260 & 4.96E{+}9 & 3.66E{+}8 & 2.946E{+}7 & 0.7941 & 3.412E{-}6\\ 
    270 & 9.29E{+}9 & 6.64E{+}8 & 4.204E{+}7 & 0.7969 & 1.992E{-}6\\ 
    280 & 1.73E{+}10 & 1.19E{+}9 & 5.998E{+}7 & 0.7996 & 1.169E{-}6\\ 
    290 & 3.16E{+}10 & 2.18E{+}9 & 8.558E{+}7 & 0.8023 & 6.888E{-}7\\ 
    300 & 5.83E{+}10 & 3.97E{+}9 & 1.197E{+}8 & 0.8048 & 4.080E{-}7\\ 
    310 & 1.06E{+}11 & 6.87E{+}9 & 1.678E{+}8 & 0.8072 & 2.428E{-}7\\ 
    320 & 1.90E{+}11 & 1.20E{+}10 & 2.346E{+}8 & 0.8094 & 1.451E{-}7\\ 
    330 & 3.38E{+}11 & 2.10E{+}10 & 3.297E{+}8 & 0.8117 & 8.713E{-}8\\ 

\end{array}\]
\end{figure}

\noindent
Note that the probability that an odd composite in the interval $(2^{k-1},2^k)$ passes the Fermat test is given by
\begin{equation*}
P(2^{k-1},2^k) := \frac{S_c(2^{k-1},2^k)}{S_p(2,^{k-1},2^k)+S_c(2,^{k-1},2^k)}.
\end{equation*}
We have directly computed $S_p(2^{k-1},2^k)$ and $S_c(2^{k-1},2^k)$ for $k\le 36$, with the latter computation aided by the formula $F(n)=\prod_{p\mid n}(p-1,n-1)$. Specifically, $S_p$ is computed directly from the available list of primes up to $2^{36}$. To compute $S_c$ we use a sieve-like procedure.
We initialize an array representing the odd numbers from $2^{k-1}$ and $2^k$ with all 1's.
For each prime $p$ to $2^k/3$, we let $m$ run over the odd numbers between $2^{k-1}/p$ and $2^k/p$.  
For each $m$, we locate $mp$ in the array, multiplying the
entry there by $\gcd(m-1,p-1)$.   At the end of the run the non-1 entries in our array correspond
to the numbers $F(n)$ for $n$ odd and composite.   Note this avoids factoring integers $n$ in $(2^{k-1},2^k)$, though a brute force method to the modest level of $2^{36}$ would have worked too.  

In Figure \ref{fig:exact}, we provide the values of $S_p(2^k)$ and $S_c(2^k)$, as well as $P(2^k)$ and $P(2^{k-1},2^k)$.

\begin{figure}[H]
  \caption{Exact values of data.} \label{fig:exact}
  \[\begin{array}{rllcc}
    k & S_p(2^{k-1},2^k) & S_c(2^{k-1},2^k) & P(2^{k-1},2^k) & P(2^k)\\
    \hline
    3 & 6 & 0 & 0 & 0\\
    4 & 18 & 2 & 1.000E{-}1 & 7.692E{-}2\\
    5 & 104 & 4 & 3.704E{-}2 & 4.478E{-}2\\
    6 & 320 & 24 & 6.977E{-}2 & 6.276E{-}2\\
    7 & 1180 & 114 & 8.810E{-}2 & 8.126E{-}2\\
    8 & 4292 & 316 & 6.858E{-}2 & 7.210E{-}2\\
    9 & 16338 & 1114 & 6.384E{-}2 & 6.605E{-}2\\
    10 & 57416 & 3056 & 5.054E{-}2 & 5.492E{-}2\\
    11 & 208576 & 10890 & 4.962E{-}2 & 5.109E{-}2\\
    12 & 780150 & 28094 & 3.476E{-}2 & 3.922E{-}2\\
    13 & 2837158 & 74528 & 2.600E{-}2 & 2.936E{-}2\\
    14 & 10673384 & 231514 & 2.123E{-}2 & 2.342E{-}2\\
    15 & 39467286 & 582318 & 1.454E{-}2 & 1.695E{-}2\\
    16 & 148222234 & 1636968 & 1.092E{-}2 & 1.254E{-}2\\
    17 & 559288478 & 4521166 & 8.019E{-}3 & 9.224E{-}3\\
    18 & 2106190104 & 11682336 & 5.516E{-}3 & 6.503E{-}3\\
    19 & 7995006772 & 33290330 & 4.147E{-}3 & 4.770E{-}3\\
    20 & 30299256236 & 88781082 & 2.922E{-}3 & 3.410E{-}3\\
    21 & 115430158810 & 230250774 & 1.991E{-}3 & 2.364E{-}3\\
    22 & 440353630422 & 628735800 & 1.426E{-}3 & 1.672E{-}3\\
    23 & 1683364186642 & 1680806136 & 9.975E{-}4 & 1.174E{-}3\\
    24 & 6448755473484 & 4408788648 & 6.832E{-}4 & 8.115E{-}4\\
    25 & 24754014371036 & 11552686982 & 4.665E{-}4 & 5.565E{-}4\\
    26 & 95132822935752 & 30756273488 & 3.232E{-}4 & 3.840E{-}4\\
    27 & 366232744269106 & 82133627362 & 2.242E{-}4 & 2.657E{-}4\\
    28 & 1411967930053822 & 215629423796 & 1.527E{-}4 & 1.820E{-}4\\
    29 & 5450257882815404 & 565834872742 & 1.038E{-}4 & 1.241E{-}4\\
    30 & 21065843780715212 & 1504267288346 & 7.140E{-}5 & 8.504E{-}5\\
    31 & 81507897575948416 & 3999812059436 & 4.907E{-}5 & 5.837E{-}5\\
    32 & 315718919767278610 & 10350692466866 & 3.278E{-}5 & 3.940E{-}5\\
    33 & 1224166825030041460 & 27472503360964 & 2.244E{-}5 & 2.682E{-}5\\
    34 & 4750936696054816476 & 72288538641772 & 1.522E{-}5 & 1.821E{-}5\\
    35 & 18454541611019193346 & 190806759987694 & 1.034E{-}5 & 1.237E{-}5\\
    36 & 71745407298862105164 & 498526567616818 & 6.949E{-}6 & 8.342E{-}6
\end{array}\]
\end{figure}

Additionally, we have estimated $P(2^k)$ in the range $30\le k \le 50$ using random sampling. More precisely, we randomly sample $\lfloor2^{k/2}\rfloor$ odd composite numbers in the interval $(2^{k-1},2^k)$, estimating $S_p(2^{k-1},2^k)$ by $$\widehat{S}_p(2^{k-1},2^k) = \int_{2^{k-1}}^{2^k}\frac{t-3}{\log t}\;\rd t = \Li(2^{2k}) - \Li(2^{2(k-1)}) - 3\big(\Li(2^{k}) - \Li(2^{k-1})\big),$$
in order to smooth out some noise from the experiment. To estimate $S_c(2^{k-1},2^k)$, we add up $F(n)-2$ for each odd composite $n$ sampled, and scale this sum by
$$\frac{2^{k-2} - \Li(2^k) + \Li(2^{k-1})}{2^{k/2}},$$
representing the ratio between the number of composites in the interval and the number of samples taken. We repeat this procedure ten times, and compute the mean, $\widehat{S}_{\text{mean}}(2^{k-1},2^k)$, and median, $\widehat{S}_{\text{median}}(2^{k-1},2^k)$, of the data. Using these statistics, we estimate $P(2^{k-1},2^k)$ by
\begin{align*}
\widehat{P}_{\text{mean}}(2^{k-1},2^k) &= \frac{\widehat{S}_{\text{mean}}(2^{k-1},2^k)}{\widehat{S}_p(2^{k-1},2^k) + \widehat{S}_{\text{mean}}(2^{k-1},2^k)},\\
\widehat{P}_{\text{median}}(2^{k-1},2^k) &= \frac{\widehat{S}_{\text{median}}(2^{k-1},2^k)}{\widehat{S}_p(2^{k-1},2^k) + \widehat{S}_{\text{median}}(2^{k-1},2^k)}.
\end{align*}
\noindent
For $30\le k\le 36$, $P(2^{k-1},2^k)$ is known, in which case we compute the relative errors, $\widehat{P}_{\text{mean}}/P-1$ and $\widehat{P}_{\text{median}}/P-1$, to get a sense of the accuracy of the experiment. 
Then we estimate $P(2^k)$ by 
$$\widehat{P}_{\text{mean}}(2^k) = \frac{\widehat{S}_{\text{mean}}(2^k)}{\widehat{S}_p(2^k) + \widehat{S}_{\text{mean}}(2^k)}$$
where
\[ \widehat{S}_{\text{mean}}(2^k) = \begin{cases} 
      S_c(2^{k-1}) + \widehat{S}_{\text{mean}}(2^{k-1},2^k) & \text{for }30\le k \le 36, \\
      S_c(2^{36}) + \sum_{j=37}^k \widehat{S}_{\text{mean}}(2^{j-1},2^j) & \text{for }37\le k \le 50,
   \end{cases}\]
and
\[ \widehat{S}_p(2^k) = \begin{cases} 
      S_p(2^{k-1}) + \widehat{S}_p(2^{k-1},2^k) & \text{for }30\le k \le 36, \\
      S_p(2^{36}) + \sum_{j=37}^k \widehat{S}_p(2^{j-1},2^j) & \text{for }37\le k \le 50.
   \end{cases}\]
\\   
\noindent   
Results of the random sampling experiment are summarized in Figures \ref{fig:known} and \ref{fig:unknown}.

One sees a negative bias in these data with the results of random sampling undershooting
the true figures.  The referee has pointed out to us that this may be due to Jensen's inequality
applied to the convex function $x/(a+x)$, so that ${\rm E}[X/(a+X)]\ge {\rm E}[X]/(a+{\rm E}[X])$.  The undershoot
may also be due to the fact that on average $F(n)$ is much larger than it is typically.
In fact, it is shown in \cite{EP} that 
on a set of asymptotic density 1, we have $F(n)=n^{o(1)}$, yet
the average behavior is $>n^{15/23}$.  The exponent $15/23$, after more recent work of
Baker and Harman \cite{BH}, can be replaced with $0.7039$.  It follows from an old conjecture of
Erd\H os on the distribution of Carmichael numbers that on average $F(n)$ behaves like $n^{1-o(1)}$.

\begin{figure}[H]
  \caption{Random sampling estimates in range where $P(2^k)$ is known.} \label{fig:known}
  \[\begin{array}{c|cc|cc|cc}
k & \widehat{P}_{\text{mean}}(2^{k-1},2^k) & \text{rel.\ err.} & \widehat{P}_{\text{median}}(2^{k-1},2^k) & \text{rel.\ err.}& \widehat{P}_{\text{mean}}(2^k) & \text{rel.\ err.}\\
\hline
30 & 5.541E{-}5 &-0.224 & 5.045E{-}5 &-0.293 & 7.319E{-}5 &-0.139\\
31 & 4.800E{-}5 &-0.022 & 3.616E{-}5 &-0.263 & 5.758E{-}5 &-0.014\\
32 & 2.706E{-}5 &-0.175 & 1.899E{-}5 &-0.421 & 3.515E{-}5 &-0.108\\
33 & 2.223E{-}5 &-0.009 & 1.248E{-}5 &-0.444 & 2.666E{-}5 &-0.006\\
34 & 1.387E{-}5 &-0.088 & 1.013E{-}5 &-0.334 & 1.721E{-}5 &-0.055\\
35 & 7.603E{-}6 &-0.265 & 6.506E{-}6 &-0.371 & 1.033E{-}5 &-0.165\\
36 & 4.433E{-}6 &-0.362 & 4.123E{-}6 &-0.407 & 6.474E{-}6 &-0.224
\end{array}\]
\end{figure}

\begin{figure}[H]
  \caption{Random sampling estimates in range where $P(2^k)$ is unknown.} \label{fig:unknown}
  \[\begin{array}{c|ccc}
k & \widehat{P}_{\text{mean}}(2^{k-1},2^k) & \widehat{P}_{\text{median}}(2^{k-1},2^k) & \widehat{P}_{\text{mean}}(2^k) \\
\hline
37 & 4.113E{-}6 & 3.675E{-}6 & 5.200E{-}6\\
38 & 4.807E{-}6 & 2.677E{-}6 & 4.908E{-}6\\
39 & 3.008E{-}6 & 1.463E{-}6 & 3.496E{-}6\\
40 & 1.519E{-}6 & 1.097E{-}6 & 2.026E{-}6\\
41 & 9.078E{-}7 & 5.697E{-}7 & 1.194E{-}6\\
42 & 7.747E{-}7 & 3.772E{-}7 & 8.822E{-}7\\
43 & 3.472E{-}7 & 2.334E{-}7 & 4.842E{-}7\\
44 & 1.968E{-}7 & 1.677E{-}7 & 2.704E{-}7\\
45 & 1.639E{-}7 & 1.687E{-}7 & 1.911E{-}7\\
46 & 1.186E{-}7 & 1.198E{-}7 & 1.372E{-}7\\
47 & 1.051E{-}7 & 6.597E{-}8 & 1.133E{-}7\\
48 & 4.076E{-}8 & 3.947E{-}8 & 5.928E{-}8\\
49 & 3.791E{-}8 & 3.213E{-}8 & 4.337E{-}8\\
50 & 2.361E{-}8 & 1.318E{-}8 & 2.865E{-}8
\end{array}\]
\end{figure}

\section*{Acknowledgments}
We gratefully acknowledge the many constructive comments of an anonymous referee.
The first-named author is grateful for support received from the Byrne Scholars Program and the James O. Freedman
Presidential Scholars program at Dartmouth College.

\bibliographystyle{amsplain}

\end{document}